\documentclass{amsart}

\usepackage{amsmath,amsthm,amssymb,amscd}
\usepackage{epic,eepic}
\usepackage{wrapfig}
\usepackage{bm}
\usepackage[usenames]{color}

\newtheorem{thm}{Theorem}[section]
\newtheorem{prop}[thm]{Proposition}
\newtheorem{lem}[thm]{Lemma}
\newtheorem{cor}[thm]{Corollary}

\theoremstyle{definition}
\newtheorem{dfn}[thm]{Definition}
\theoremstyle{remark}
\newtheorem{eg}[thm]{Example}
\theoremstyle{remark}
\newtheorem{rem}[thm]{Remark}

\newcommand{\Hom}{{\rm Hom}}

\newcommand{\Z}{\mathbb{Z}}

\title{Generalization of neighborhood complexes}
\author{Takahiro Matsushita}

\begin{document}
\maketitle
\begin{abstract} 
We introduce the notion of r-neighborhood complex for a positive integer $r$, which is a natural generalization of Lov$\acute{\rm a}$sz' neighborhood complex. The topologies of these complexes give some obstructions of the existence of graph maps. We applied these complexes to prove the nonexistence of graph maps about Kneser graphs. We prove that the fundamental groups of $r$-neighborhood complexes are closely related to the $(2r)$-fundamental groups defined in \cite{Mat2}.

\end{abstract}

\section{Introduction}

Neighborhood complexes of graphs were originally defined by Lov$\acute{\rm a}$sz in \cite{Lov} in the context of graph coloring problem. He proved that the connectivity of neighborhood complexes give the lower bounds of chromatic numbers of graphs, and determined the chromatic numbers of Kneser graphs. After that, many people researched about neighborhood complexes and related cell complexes, see \cite{BK}, \cite{Cso}, \cite{Sch}, and \cite{Ziv} for example. In this paper, we introduce the $r$-neighborhood complex $\mathcal{N}_r(G)$, which is the natural generalization of the neighborhood complexes $\mathcal{N}(G)$, and show that we can obtain the obstructions of the existence of graph maps from it. The first main theorem is the following. The all necessary definitions will be given later in this paper.

\vspace{2mm}
\noindent {\bf Main Theorem 1.}  (Corollary 4.6) Let $n,k$ be positive integers and $r$ a nonnegative integer satisfying $k-1 = r(n-2k)$. If $G$ is a graph with $\sharp V(G) < \sharp V(K_{n,k})$ and the odd girth of $G$ is equal to $K_{n,k}$, then there are no graph maps from $K_{n,k}$ to $G$.\qed

\vspace{2mm}
In \cite{Mat1}, the author proved that the fundamental groups of neighborhood complexes are isomorphic to the subgroup of 2-fundamental groups of graphs, called the even part. In \cite{Mat2}, the author established the notion of $r$-fundamental groups. Then the following hold.

\vspace{2mm}
\noindent {\bf Main Theorem 2.} (Theorem 5.2) Let $r$ be a positive integer. Let $(G,v)$ be a based graph where $v$ is not isolated. Then there is a natural isomorphism $\pi_1(\mathcal{N}_r(G),v) \cong \pi_1^{2r}(G,v)_{ev}$. \qed

\vspace{2mm} From this theorem and combined with the result of \cite{Mat2}, we have the following theorem.

\vspace{2mm}
\noindent {\bf Main Theorem 3.} (Corollary 5.3) Let $G$ be a graph whose chromatic number is greater than $2$, and $r$ a nonnegative integer. If there is a graph map from $G$ to $C_m$, then $H_1(\mathcal{N}_{i}(G) ; \mathbb{Z})$ has $\Z$ as a direct summand for any $i \leq r$. \qed

\section{Definitions}

In this section, we fix some definitions and review some facts we need in this paper.

\vspace{2mm}
\noindent {\bf Graphs :} A {\it graph} is a pair $(V,E)$ where $V$ is a set and $E$ is a subset of $V \times V$ such that $(x,y) \in E$ implies $(y,x) \in E$. We do not assume that $V$ is finite. For a graph $G=(V,E)$, we write $V(G)$ for the set $V$ and called the {\it vertex set of $G$} and $E(G)$ for the set $E$. A graph map from $G$ to $H$ is a map $f:V(G) \rightarrow V(H)$ such that $(f\times f)(E(G)) \subset E(H)$. We write $C_n$ for the cycle graph with $n$-vertices. Namely, $V(C_n)= \Z /n\Z$ and $E(C_n) = \{ (i,i+1), (i+1,i) \; | \; i \in \Z /n\Z\}$. The odd girth of a graph $G$ is a number $\inf \{2m+1\; |$ There is a graph map $C_{2m+1} \rightarrow G\}$ and, in this paper, denoted by $g_0(G)$. Remark that $g_0(G) < g_0(H)$, there is no graph map from $G$ to $H$.

\vspace{2mm}
\noindent {\bf $\Z_2$-space :} To obtain the obstruction of the existence of graph maps, we need 1st Stiefel-Whitney classes of double covers. We review the basic properties of the 1st Stiefel-Whitney classes following \cite{Koz}. For further studies about the Borsuk-Ulam theorems and characteristic classes are found in \cite{Matousek}, \cite{MS}. For the definition and basic properties for singular homology and cohomology are found in \cite{Hat} or \cite{Spa}, for example.

 A {\it $\mathbb{Z}_2$-space} is a topological space $X$ with an involution on $X$. We write $\overline{X}$ for the orbit space of $X$. In this paper, we say that a $\mathbb{Z}_2$-space $X$ is {\it free} if the natural quotient map $X \rightarrow X/\mathbb{Z}_2$ is a double covering. For a free $\mathbb{Z}_2$-space $X$, we write $w_1(X) \in H^1(\overline{X};\mathbb{Z}_2)$ for the 1st Stiefel-Whitney class of the double covering $X \rightarrow \overline{X}$. $w_1(X)$ is natural with respect to $\mathbb{Z}_2$-maps.
 We define the $\mathbb{Z}_2$-height ${\rm ht}_{\Z_2}(X)$ of a free $\mathbb{Z}_2$-space $X$ by the number $\inf \{ n\geq 0 \; | \; w_1(X)^n \neq 0.\}$. Remark that if ${\rm ht}_{\mathbb{Z}_2}(X) > {\rm ht}_{\Z_2}(Y)$, then there are no $\Z_2$-maps from $X$ to $Y$.

\vspace{2mm}
\noindent {\bf Posets :} A {\it chain} of a poset $P$ is a subset $c$ of $P$ such that $c$ becomes totally ordered set with the ordering induced by $P$. The {\it order complex} of a poset $P$ is the simplicial complex $\Delta (P) = \{ c\subset P \; | \; c$ is a finite chain of $P.\}$. In this paper, we say that a poset map $f:P \rightarrow Q$ is a homotopy equivalence if the map $|\Delta (P)| \rightarrow |\Delta (Q)|$ is a homotopy equivalence.

\section{$r$-neighborhood complexes}

In this section, we give the definition of $r$-neighborhood complexes of graphs, and state the criterion to obtain the nonexistence of graph maps. The case $r=1$ is known as the neighborhood complex. The method to obtain obstructions is essentially the same of the case of neighborhood complexes given in \cite{BK}, although there are some improvements (see Remark 3.5).

Let $G$ be a graph. For $v \in V(G)$, we write $N(v)$ or $N^G (v)$ for the subset $\{ w \in V(G) \; | \; (v,w) \in E(G)\}$. We define the set $N_r(v)$ or $N_r^G(v)$ for a positive integer $r$ as follows. First put $N_0(v)= \{ v\}$. If $N_{r-1}$ is defined, then $N_r(v) = \bigcup_{w \in N_{r-1}(v)} N(w)$.

\begin{dfn}
Let $G$ be a graph and $r$ a positive integer. Then we write $\mathcal{N}_r(G)$ for the abstract simplicial complex $\mathcal{N}_r(G) = \{ \sigma \subset V(G) \; | \; \sharp \sigma < +\infty \textrm{ and there is $v\in V(G)$ with $\sigma \subset N_r(v).$}\} $ and called the {\it $r$-neighborhood complex of $G$}.
\end{dfn}

Remark that $\mathcal{N}_1(G)$ is known as the neighborhood complexes defined by Lov$\acute{\rm a}$sz

The next is the generalization of $\Hom (K_2,G)$.

\begin{dfn}
Let $G$ be a graph and $r$ a positive integer. Then we write $\mathcal{B}_r(G)$ as the poset $\{ (A,B) \; | \; A$ and $B$ are nonempty finite subsets of $V(G)$ and for each $x\in A$ and $y\in B$ we have that $y \in N_r(x).\}$. The ordering of $\mathcal{B}_r(G)$ is defined by that $(A,B) \leq (A',B')$ if and only if $A\subset A'$ and $B\subset B'$.
\end{dfn}

For a simplicial complex $K$, we write $FK$ for its face poset. $\mathcal{B}_r(G)$ is a subposet of $F\mathcal{N}_r(G) \times F\mathcal{N}_r(G)$.

\begin{lem}
The first projection $\mathcal{B}_r(G) \rightarrow F\mathcal{N}_r(G)$, $(A,B) \mapsto A$ is a homotopy equivalence.
\end{lem}
\begin{proof}
We write $p:\mathcal{B}_r(G) \rightarrow F\mathcal{N}_r(G)$ for the first projection. Let $\sigma \in F\mathcal{N}_r(G)$. Put $U_\sigma = p^{-1}(F\mathcal{N}_r(G)_{\geq \sigma})$. Namely $U_\sigma = \{ (\sigma ' ,\tau') \in \mathcal{B}_r(G)\; | \; \sigma \subset \sigma'\}$. Define the map $c:U_\sigma \rightarrow U_\sigma$ by $c(\sigma',\tau')= (\sigma ,\tau')$. Then $c$ is a descending closure operator of $U_\sigma$, and hence $c(U_\sigma)$ is a deformation retract of $U_\sigma$ (see \cite{Koz}). From the next lemma, we have that $c(U_\sigma)= \{ (\sigma,\tau') \in \mathcal{B}_r(G)  \}$ is contractible, and hence $U_\sigma$ is contractible. By Quillen's fiber lemma A (for its statement and its proof, see \cite{Bjo}, \cite{Qui1}, or the case of finite posets are found in \cite{Bar} and \cite{Koz}), we have that the first projection $p:\mathcal{B}_r(G) \rightarrow F\mathcal{N}_r(G)$ is a homotopy equivalence.
\end{proof}

\begin{lem}
Let $P$ be a poset and suppose there is $x_0 \in P$ such that for any $x \in P$, $P_{\geq x} \cap P_{\geq x_0}$ has the minimum. Then $P$ is contractible.
\end{lem}
\begin{proof}
Let $c: P \rightarrow P$ be a map defined by $c(x)= \min (P_{\geq x_0} \cap P_{\geq x})$. Then $c$ is an ascending closure operator and hence $P$ and $c(P)$ have the same homotopy type. Since $c(P)$ has the minimum $x_0$, we have that $c(P)$ is contractible.
\end{proof}

\begin{rem} Suppose that the graph $G$ is finite. In this case, by the strong type of Quillen's theorem A proved in \cite{Bar} and the above proof of Lemma 3.3, we have that the first projection $\mathcal{B}_r(G) \rightarrow F\mathcal{N}_r(G)$ induces a simple homotopy equivalence. Hence the order complex of $\mathcal{B}_r(G)$ and the $r$-neighborhood complex $\mathcal{N}_r(G)$ have the same simple homotopy type. The case $r=1$ is already proved by Kozlov in \cite{Koz2}.
\end{rem}

\begin{rem} By the same argument of the proof of Lemma 3.3, we have the following somewhat interesting statement: {``\it Let $\varphi:P \rightarrow Q$ be an antitone map and $P_\varphi$ a subposet $\{ (x,y) \in P \times Q \; | \; y \leq \varphi(x)\}$ of $P \times Q$. Then the first projection $P_\varphi \rightarrow P$ induces a homotopy equivalence. Moreover, if $P$ and $Q$ are finite, then this is a simple homotopy equivalence.''} This is a strong version of Lemma 4.3 of \cite{Sch2}.
\end{rem}

$\mathcal{B}_r(G)$ has a natural involution $\tau : \mathcal{B}_r(G) \rightarrow \mathcal{B}_r(G)$, $(x,y) \mapsto (y,x)$. If $r$ is odd and the odd girth $g_0(G)$ is greater than $r$, then $\mathcal{B}_r(G)$ is a free $\mathbb{Z}_2$-space. The following theorem is the criterion to obtain the obstructions from the topologies of $\mathcal{N}_r(G)$.

\begin{thm} Let $r$ be a positive odd integer, $G$ and $H$ graphs whose odd girths are greater than $r$. If ${\rm ht}_{\mathbb{Z}_2}(\mathcal{B}_r(G)) > {\rm ht}_{\mathbb{Z}_2}(\mathcal{B}_r(H))$, then there is no graph map from $G$ to $H$.
\end{thm}

For example, for a positive odd integer $r$, if $\mathcal{N}_r(H)$ and is $k$-connected and $g_0(G)> r$, and $\mathcal{N}_r(G)$ has the dimension smaller than $k$, then there is no graph map from $G$ to $H$.

\begin{eg}
Let $r$ be a positive integer. The first obvious example of $r$-neighborhood complex is $\mathcal{N}_r(C_{r+2})$. It is easy to see that for any $v \in V(C_{r+2})$, $N_{r}^{C_{r+2}}(v)= V(C_{r+2}) \setminus \{ v\}$. Hence $\mathcal{N}_r(C_{r+2})$ is homeomorphic to $S^r$. Hence ${\rm ht}_{\mathbb{Z}_2}(\mathcal{B}_r(C_{r+2}))= r.$ Hence for a graph $G$ with $g_0 (G)=r+2$, we have that ${\rm ht}_{\Z_2}(\mathcal{B}_r(G))\geq r$.
\end{eg}

\begin{eg}
Let $r$ be a positive integer and $m$ an odd integer greater than $2r$. Then $\mathcal{N}_r(C_m)$ is homotopy equivalent to $S^1$. Moreover $\mathcal{N}_r(C_m)$ collapses to the boundary of the $m$-gon. To prove this, we use discrete Morse theory. For studying discrete Morse theory, refer \cite{For} or \cite{Koz}, for example. The terminologies we use here is due to \cite{Koz}.

We prove this by induction on $r$. The case $r=1$ is obvious. Suppose $r\geq 2$. Then it is sufficient to prove that $\mathcal{N}_r(C_m)$ collapses to $\mathcal{N}_{r-1}(C_m)$ by the inductive hypothesis. Since $F \mathcal{N}_{r-1}(G)$ is downward closed in $F\mathcal{N}_r(G)$, it is sufficient to prove that we construct acyclic matching on $F\mathcal{N}_r(G) \setminus F\mathcal{N}_{r-1}(G)$ which has no critical points. 

For any $x \in \Z/m\Z$, we write $P_x$ for the poset $\{ \sigma \in F\mathcal{N}_r(C_m) \; | \; \{ x,x+2r\} \subset \sigma \}$. Then $F\mathcal{N}_{r}(C_m) \setminus F\mathcal{N}_{r-1}(C_m) = \coprod_{x \in \Z/m\Z} P_x$. The matching $M_x$ on $P_x$ defined as follows.

\begin{itemize}
\item $(\{ x,x+4, x+6, \cdots , x+2r\} , \{ x,x+2,\cdots,x+2r\}) \in M_x$.
\item Let $\sigma \in P_x$ such that $\sigma \neq \{ x,x+4, x+6, \cdots , x+2r\} , \{ x,x+2,\cdots,x+2r\}$. Put $d(\sigma)= \max \{ 2i\; | \; x+2i \not\in \sigma\}$. If $x+d(\sigma)-2 \in \sigma$, then $(\sigma \setminus \{ x+d(\sigma) -2\},\sigma) \in M_x$.
\end{itemize}
Then $M_x$ is a matching on $P_x$ having no critical points. We want to show that the matching $M_x$ is acyclic. We write $x<_{M_x} y$ for $(x,y) \in M_x$. A sequence of $a_0,\cdots ,a_n$ of points of $P_x$ is said to be admissible if for any $i \in \{ 1,\cdots ,n\}$, we have that $a_{i-1} > a_i$ or $a_{i-1} <_{M_x} a_i$.
\begin{itemize}
\item[(1)] If $\tau < \sigma$, then $d(\tau) \geq d(\sigma)$.
\item[(2)] If $(\tau,\sigma) \in M_x$, then $d(\tau)=d(\sigma)$.
\end{itemize}
Hence we have the following.
\begin{itemize}
\item[(3)] For $\sigma,\tau \in P_x$, if $d(\tau)>d(\sigma)$, then there is no admissible sequence from $\tau$ to $\sigma$.
\end{itemize}
By the definition of $M_x$, we have the following.
\begin{itemize}
\item[(4)] For $\tau,\sigma \in P_x$ and suppose there is $i$ with $2i < d(\sigma)-2$ and $x+2i \in \sigma \setminus \tau$. Then there is no admissible sequence from $\tau$ to $\sigma$.
\end{itemize}
By (3) and (4), we have the following.
\begin{itemize}
\item[(5)] Let $\sigma,\tau \in P_x$ with $\tau < \sigma$. If there is an admissible sequence from $\tau$ to $\sigma$, then $(\tau,\sigma) \in M_x$.
\end{itemize}
(5) implies that $M_x$ is acyclic. Indeed, if there is a admissible sequence
$$a_0 > b_0 <_{M_x} a_1 >b_1 <_{M_x} \cdots <_{M_x} a_{n} > b_{n} <_{M_x} a_0$$
Then for any $i \in \{ 0,1,\cdots ,n\}$, there is an admissible sequence from $b_i$ to $a_i$, and hence we have that $(b_i,a_i) \in M_x$. Since $(b_i,a_{i+1}) \in M_x$, we have that $a_0 = a_1 = \cdots =a_n$ and $b_0 =b_1 = \cdots =b_n$. Hence $M_x$ is acyclic. Therefore, we have that $M= \coprod_{x \in \Z/m\Z} M_x$ is an acyclic matching on $F\mathcal{N}_r(C_m) \setminus F\mathcal{N}_{r-1}(C_m)$.

This completes the proof of the fact $\mathcal{N}_r(C_m)$ collapses to the boundary of the $m$-gon if $r$ is a positive integer and $m$ is a positive odd integer greater than $2r$. Therefore, if a graph $G$ having a graph map from $G$ to $C_m$ where $m$ is an odd integer greater than $2r$, then ${\rm ht}_{\Z_2}(\mathcal{B}_r(G)) \leq 1$.
\end{eg}

\section{Kneser graphs}

In this section, we present an application of the theory of $r$-neighborhood complexes to Kneser graphs, and prove Main Theorem 1.

Recall that the Kneser graph $K_{n,k}$ for positive numbers $n$ and $k$ is a graph defined by $V(K_{n,k})= \{ A \subset [n] \;| \; \sharp A = k\}$ and $E(K_{n,k})= \{ (A,B) \; | \; A \cap B = \emptyset \}$. The Kneser conjecture is that the chromatic number of $K_{n,k}$ is $n-2k +2$ and was proved by Lovasz in \cite{Lov} using neighborhood complexes.

To prove Main Theorem 1, the following key observation is needed.

\begin{lem}
Let $A,B \in V(K_{n,k})$, $s$ a positive integer. Then the following are equivalent.
\begin{itemize}
\item[(1)] $B \in N_{2s}(A)$.
\item[(2)] $\sharp (A \setminus B) \leq s(n-2k)$.
\end{itemize}
\end{lem}
\begin{proof}
We prove $(1) \Rightarrow (2)$. First we prove the case $s=1$. Since $B \in N_{2}(A)$, there is $C \in V(K_{n,k})$ such that $A \cap C = \emptyset$ and $B \cap C = \emptyset$. Hence we have that $\sharp (A \setminus B) \leq \sharp ([n] - (B \sqcup C)) = n-2k$. Next suppose $s >1$. Since $B \in N_{2s}(A)$, there is a sequence $A_0 ,\cdots ,A_s$ where $A_0 =A$ and $A_s =B$ and $A_i \in N_2(A_{i-1})$ for $i=1,\cdots ,s$. Then we have that $\sharp (A \setminus B) = \sum _{i=1}^n \sharp (A_{i-1} \setminus A_i) \leq s(n-2k)$.

Next we prove $(2) \Rightarrow (1)$. Since the case $n \leq 2k$ is trivial, we assume $n>2k$. We prove the induction on $s$. First we prove the case $s=1$. $\sharp (A \setminus B) \leq n-2k$ implies that $\sharp (A \cup B) = \sharp ((A \setminus B) \sqcup A) \leq n-k$. Hence there is a $k$-subset $C$ of $[n]$ such that $A \cap C = \emptyset$ and $B \cap C = \emptyset$. This implies $B \in N_2(A)$ in $K_{n,k}$. Next suppose $s>1$ and $\sharp (A \setminus B) \leq s(n-2k)$. By the inductive hypothesis, we can assume $(s-1)(n-2k) <\sharp (A \setminus B)$. Since $s>1$ and $n-2k > 0$, we have that $2 \leq \sharp (A \setminus B)$. Since $\sharp (A \setminus B) = \sharp (B \setminus A)$, there are two element subsets $X \subset A \setminus B$ and $Y \subset B \setminus A$. Put $C= (A \setminus X) \cup Y $. By the case $s=1$, we have that $C \in N_2(A)$. By the fact $\sharp (B \setminus C) \leq (s-1)(n-2k)$ and the inductive hypothesis,  we have that $B \in N_{2(s-1)}(C)$. Hence we have that $B \in N_{2s}(A)$.
\end{proof}

For a real number $x$, we write $\lceil x \rceil$ for the minimum integer which is not smaller than $x$. The following corollary is well-known. But we give the proof for convenient.

\begin{cor}
Let $n,k$ be positive integers with $n>2k$. Then we have that $g_0(K_{n,k}) = 2 \lceil \dfrac{k}{n-2k}\rceil +1$.
\end{cor}
\begin{proof}
Suppose that there is a graph map $C_{2m+1} \rightarrow K_{n,k}$ where $m$ is a nonnegative integer. Then there are $A,B \in V(K_{n,k})$ such that $B \in N_{2m}(A)$ and $A \cap B = \emptyset$. Hence we have that $k = \sharp (A \setminus B) \leq m(n-2k)$. Therefore we have $m\geq \lceil k/(n-2k) \rceil$. Hence we have that $g_0(K_{n,k})\geq 2 \lceil k/(n-2k) \rceil +1$. Next suppose that $m = \lceil k/(n-2k) \rceil$ and let $A, B \in K_{n,k}$ with $A \cap B = \emptyset$. Since $\sharp (A \setminus B) =k \leq m(n-2k)$, we have that $B \in N_{2m}(A)$. Hence there is a graph map $C_{2m +1} \rightarrow K_{n,k}$. This completes the proof of $g_0(K_{n,k})= 2\lceil k/(n-2k)\rceil +1$.
\end{proof}

\begin{prop}
Let $n,k$ be positive integers and $r$ a nonnegative integer satisfying $n>2k$ and $k-1 = (n-2k)r$. Then $g_0(K_{n,k})= 2r+3$ and $N_{2r+1} (A) = V(K_{n,k}) \setminus \{ A\}$ for any $A \in V(K_{n,k})$.
\end{prop}
\begin{proof} By the equation
$$g_0(K_{n,k}) = 2\lceil \frac{k}{n-2k} \rceil +1 = 2 \lceil r + \frac{1}{n-2k}\rceil +1 = 2(r+1)+1 =2r+3,$$
we have that $g_0(K_{n,k}) = 2r+3$. Let $A,B \in V(K_{n,k})$ with $A \neq B$, and $a \in A \setminus B$. Since $n\leq  2k+1$, there is a subset $C \subset [n]$ with $C \cap B = \emptyset$ and $a \in C$. Then we have that $\sharp (A \setminus C) \leq k-1 =r(n-2k)$. This implies $C \in N_{2r}(A)$, and hence we have that $B \in N_{2r+1}(A)$.
\end{proof}

\begin{cor}
Let $n,k$ be positive integers and $r$ a nonnegative integer satisfying that $n>2k$ and $k-1= (n-2k)r$. Then the $(2r+1)$-neighborhood complex $\mathcal{N}_{2r+1}(K_{n,k})$ is homeomorphic to the sphere whose dimension is equal to $\dbinom{n}{k} -2$.
\end{cor}

\begin{rem}
For $n,k,r$ satisfying the above corollary, by Proposition 4.3, we can show that $\mathcal{B}_{2r+1}(K_{n,k})$ is $\mathbb{Z}_2$-homeomorphic to the sphere whose dimension is $\dbinom{n}{k}-2$ with the antipodal action. The proof of this fact is almost the same proof of ${\rm Hom}(K_2,K_n) \approx S^{n-2}$, which is found in \cite{BK}, for example.
\end{rem}

The following corollary includes the claim of Main Theorem 1.

\begin{cor}
Let $n,k$ be positive integers and $r$ a positive integer satisfying $k-1= r(n-2k)$. Let $G$ be a graph satisfying the followings:
\begin{itemize}
\item $g_0(G) > 2r+1$.
\item ${\rm ht}_{\Z_2} (\mathcal{B}_{2r+1}(G)) < \dbinom{n}{k}-2$.
\end{itemize}
Then there are no graph maps from $K_{n,k}$ to $G$. In particular, if $g_0(G)>2r+1$ and $\sharp V(G) < \sharp V(K_{n,k})$, then there are no graph maps from $K_{n,k}$ to $G$.
\end{cor}

\section{Relations with $r$-fundamental groups}

In this section, we investigate the relation between the $(2r)$-fundamental group and the fundamental groups of $r$-neighborhood complexes. The basic theory of $r$-covering maps and $r$-fundamental groups is constructed in \cite{Mat2}. The $r$-fundamental groups are closely related to the existence of a graph map to odd cycles. First we consider the relation with $r$-covering maps.

\begin{lem}
Let $p:G\rightarrow H$ be an $(2r)$-covering map. Then $\mathcal{N}_r(G) \rightarrow \mathcal{N}_r(H)$ is a covering map.
\end{lem}

The proof of this lemma is very similar of the proof of Proposition 3.5 of \cite{Mat1}, and is abbreviated.

The following theorem is Main Theorem 2. The case $r=1$ is already known in \cite{Mat1}.

\begin{thm}
Let $(G,v)$ be a based graph with $N(v) \neq \emptyset$. Then there is a natural isomorphism $\pi_1(\mathcal{N}_r(G),v) \cong \pi_1^{2r}(G,v)_{\rm ev}$.
\end{thm}

Recall that $\pi_1^r(G,v)_{\rm ev}$ is the subgroup of $\pi_1^r(G,v)$ whose index is 1 or 2. Suppose $G$ is connected and $\chi (G) >2$. If there is a graph map from $G$ to $C_{2m+1}$, we have that there is a surjective graph homomorphism from $\pi_1^{2m}(G,v)$ to $\Z$ (see \cite{Mat2}). By the above theorem, we have the following.

\begin{cor}
Let $G$ be a graph with $\chi (G) >2$ and $r$ a nonnegative integer. If there is a graph map from $G$ to $C_m$, then $H_1(\mathcal{N}_{i}(G) ; \mathbb{Z})$ has $\Z$ as a direct summand for any $i \leq r$.
\end{cor}

To prove Theorem 5.2, first we review the edge path group of simplicial complexes. These constructions are classically known, and found in \cite{Spa} for example.

Let $(K,v)$ be a based simplicial complex. A finite sequence $(v_0,\cdots ,v_n)$ such that $\{ v_{i-1},v_i\} \in K$ for any $i= 1,\cdots ,n$ and $v_0=v =v_n$ is called an {\it edge loop}. We write $L(K,v)$ for the set of all edge loops of $(K,v)$. We write $\simeq$ for the equivalence relation generated by followings.
\begin{itemize}
\item[(1)] If there is $i \in \{ 1,\cdots ,n\}$ such that $v_i=v_{i-1}$, then $(v_0,\cdots ,v_n) \simeq (v_0,\cdots,v_{i-1},v_{i+1},\cdots, v_n)$.
\item[(2)] If there is $i \in \{ 1,\cdots ,n-1\}$ such that $\{ v_{i-1},v_i,v_{i+1}\} \in K$, then $(v_0,\cdots ,v_n) \simeq (v_0,\cdots ,v_{i-1},v_{i+1}, \cdots ,v_n)$.
\end{itemize}
For edge loops $\gamma= (v_0,\cdots ,v_n)$ and $\gamma'=(w_0,\cdots ,w_m)$ of $(K,v)$, we define the composition $\gamma' \cdot \gamma = (v_0,\cdots ,v_n,w_1,\cdots ,w_m)$. Define $E(K,v)$ for the quotient set $L(K,v)/\simeq$. Then $E(K,v)$ forms a group with composition of edge loops. We define the map $\overline{\Phi} :E(K,v)  \rightarrow \pi_1(|K|,v)$ as following. For $(v) \in L(K,v)$, we define $\Phi (v)$ by the constant loop $I \rightarrow |K|$, $t \mapsto v$ for $t \in I$. For $(v_0,\cdots ,v_n) \in L(K,v)$ for $n\geq 1$, we define $\Phi(v_0,\cdots ,v_n)$ by $\Phi(i/n)= v_i$ $(i=0,\cdots,n)$ and $\Phi((i-1+t)/n)= (1-t)v_{i-1}+tv_i \in |K|$ for $t \in I$. Then this induces a map $\overline{\Phi}: E(K,v) \rightarrow \pi_1(|K|,v)$, $[(v_0,\cdots ,v_n)] \mapsto [\Phi(v_0,\cdots ,v_n)]$. It is easy to see that $\overline{\Phi}$ is a group homomorphism and natural with respect to basepoint preserving simplicial maps. The following lemma is a classical result (see Theorem 16 and Corollary 17 of Section 3.6 in \cite{Spa}), and the proof is abbreviated.

\begin{lem}$\overline{\Phi} : E(K,v) \rightarrow \pi_1(|K|,v)$ is an isomorphism.
\end{lem}

Hence to prove Theorem 5.2, what we must prove is that $E(\mathcal{N}_r(G),v)$ is naturally isomorphic to $\pi_1^{2r}(G,v)$. First we construct a homomorphism $\overline{U} : E(\mathcal{N}_r(G),v) \rightarrow \pi_1^{2r}(G,v)_{\rm ev}$. For any edge loop $\gamma =(v_0,\cdots ,v_n)$ of $(\mathcal{N}_r(G),v)$, let a loop $U(\gamma):L_{2rn} \rightarrow G$ such that $U(\gamma)(2ri)=v_i$. Such $U(\gamma)$ is not unique but unique up to $(2r)$-homotopy. Moreover, if $\gamma \simeq \gamma'$ for loops of $(|K|,v)$, then it is easy to see that $U(\gamma) \simeq _{2r} U(\gamma')$ in $(G,v)$. Indeed, the case $\gamma$ and $\gamma'$ satisfy the condition (1) of the previous page, $U(\gamma) \simeq_{2r} U(\gamma')$ is obvious. Suppose that $\gamma$ and $\gamma'$ satisfy the condition (2). Then we can write $\gamma =(v_0,\cdots ,v_n)$ and $\gamma' =(v_0,\cdots ,v_{i-1},v_{i+1}, \cdots ,v_n)$ where $\{ v_{i-1},v_i,v_{i+1}\} \in \mathcal{N}_r(G)$. Then there is $w \in V(G)$ with $\{ v_{i-1},v_i,v_{i+1}\} \subset N_r(w)$. Then there is a graph map $\varphi :L_{2rn} \rightarrow G$ such that $\varphi(2rj)=v_j$ for $j=0,1,\cdots ,n$ and $\varphi(r(2i-1))=w=\varphi(r(2i+1))$ and $\varphi(2ri+j)=\varphi(2ri-j)$ for $j=0,1,\cdots ,r$. Then $U(\gamma) \simeq_{2r} \varphi$. Since $\varphi$ is $(2r)$-homotopic to a loop $\psi : L_{2r(n-1)} \rightarrow G$ such that $\psi(x)=\varphi(x)$ for $x \leq r(2i-1)$ and $\psi(x)=\varphi(x+2r)$ for $x \geq r(2i-1)$. Since $\psi (2rx)=v_x$ for $x \leq i$ and $\psi(2rx)=v_{x+1}$ for $x > i$, we have that $\psi \simeq_{2r} U(\gamma')$. Therefore we have that $U(\gamma) \simeq U(\gamma')$. Hence the map $\overline{U}:E(\mathcal{N}_r(G),v) \rightarrow \pi_1^{2r}(G,v)$ is well-defined. $\overline{U}$ is obviously a group homomorphism and natural.

Next we construct the inverse $\overline{V}: \pi_1^{2r}(G,v)_{ev} \rightarrow E(G,v)$. We write $L_2(G,v)$ for the set of all loops of $(G,v)$ with even length. For a loop $\gamma : L_{2m} \rightarrow G$ of $(G,v)$, we define an edge loop $V(\gamma)$ of $(\mathcal{N}_r(G),v)$ by $V(\gamma)=(\gamma (0),\gamma(2r),\cdots, \gamma(2r[m/r]),v)$, where $[-]$ is the Gauss symbol. Since $\gamma(2r(i-1)),\gamma (2ri) \in N_{r}(\gamma (r(2i-1)))$ and $\gamma (2r[m/r]),v \in N_{m-r[m/r]}(r[m/r]+m)$ and $m-r[m/r] \leq r$, we have that $V(\gamma)$ is an edge loop of $(\mathcal{N}_r(G),v)$. We want to show that $V$ induces a map $\overline{V} : \pi_1^{2r}(G,v)_{ev} \rightarrow E(G,v)$. Namely, it is sufficient to show that if $\gamma \simeq_{2r} \gamma'$, then $V(\gamma) \simeq V(\gamma')$. From the definition of $(2r)$-fundamental group, we can assume $\gamma$ and $\gamma'$ satisfy one of the following conditions. We write $l(\gamma)$ for the length of a path $\gamma$.

\begin{itemize}
\item[(i)] $l(\gamma')=l(\gamma)+2$ and there is $x \in \{ 1,2,\cdots, l(\gamma)-1\}$ such that $\gamma(i)=\gamma'(i)$ $(i\leq x)$ and $\gamma(i)=\gamma'(i+2)$ $(i\geq x)$.
\item[(ii)] $l(\gamma)=l(\gamma')$ and there is $x \in \{ 1,\cdots ,l(\gamma)\}$ such that $\gamma(i)=\gamma'(i)$ for any $i \neq x,x+1,\cdots ,x+r-2$.
\end{itemize}

First we consider the case (i). Let $\gamma : L_{2m} \rightarrow G$ and $\gamma' : L_{2m+2} \rightarrow G$ be loops of $(G,v)$. Then $V(\gamma) \simeq V(\gamma')$ since
$\{ \gamma (2ri),\gamma (2r(i+1)),\gamma' (2r(i+1))\} \in N_r(\gamma(r(2i+1)))$, $\{ \gamma(2ri),\gamma'(2ri),\gamma' (2r(i+1))\} \in N_r(\gamma'(r(2i+1)))$.

Next we consider the case (ii). we let loops $\gamma ,\gamma' :L_{2m} \rightarrow G$ and suppose there is $x \in \{ 0,1,\cdots ,2m-1\}$ such that $\gamma(i)=\gamma'(i)$ for any $i\neq x,x+1,\cdots ,x+r-2$. Since we prove the case (i), we can suppose that $r$ divides $m$ and $2r$ divides $x$ or $x-1$. The case $2r$ divides $x-1$, we have that $V(\gamma)=V(\gamma')$. Suppose that $2r$ divides $x$. Then $\{ \gamma (x),\gamma' (x),\gamma(x-2r) \} \in N_r(\gamma (x-r))$. What we must prove is that the edge path $(\gamma(x),\gamma'(x),\gamma(x+2r))$ is homotopic to the edge path $(\gamma(x),\gamma(x+2r))$. Since $\{ \gamma'(x),\gamma'(x+2r-2),\gamma(x+2r)\} \subset N_r(\gamma (x+r))$, we have that $(\gamma(x),\gamma'(x),\gamma(x+2r)) \simeq (\gamma(x),\gamma'(x),\gamma(x+2r-2),\gamma(x+2r))$. Since $\{ \gamma(x),\gamma'(x),\gamma(x+2r-2)\} \subset N_r(\gamma(x+r-1))$, we have that $(\gamma(x),\gamma'(x),\gamma(x+2r-2),\gamma(x+2r)) \simeq (\gamma(x),\gamma(x+2r-2),\gamma(x+2r))$. Since $\{ \gamma(x),\gamma(x+2r-2),\gamma(x+2r)\} \subset N_r(\gamma(x+r)),$ we have that $(\gamma(x),\gamma(x+2r-2),\gamma(x+2r))\simeq (\gamma(x),\gamma(x+2r))$. Hence we have that $(\gamma(x),\gamma'(x),\gamma(x+2r)) \simeq (\gamma(x),\gamma(x+2r))$. Therefore we have that $V$ induces a map $\overline{V}:\pi_1^{2r}(G,v) \rightarrow \pi_1(\mathcal{N}_r(G),v)$.

It is easy to see that $\overline{V}$ is the inverse of $\overline{U}$. This completes the proof of Theorem 5.2.

\end{document}